\documentclass[a4j,12pt]{article}
\usepackage{amssymb}
\usepackage{latexsym}
\usepackage{graphicx} 
\usepackage{float}    
\usepackage{amsmath,amsthm}
\usepackage{color}

\newtheorem{theorem}{Theorem}[section]
\newtheorem{lemma}[theorem]{Lemma}
\newtheorem{proposition}[theorem]{Proposition}
\newtheorem{definition}[theorem]{Definition}
\newtheorem{example}{Example}
 \newtheorem{rem}[theorem]{Remark}

\newcommand{\clco}{{\rm cl}\,{\rm co}\,}

\newcommand{\epir}[1]{{\rm epi}_{#1}\,}

\newcommand{\tgcl}{\tilde{g}^{\rm cl}}

\newcommand{\dom}{{\rm dom}\,}
\newcommand{\epi}{{\rm epi}\,}

\def\real{\RRR}

\def\integer{\ZZZ}
\DeclareSymbolFont{lettersA}{U}{txmia}{m}{it}
 \DeclareMathSymbol{\FFF}{\mathord}{lettersA}{'206}
 \DeclareMathSymbol{\NNN}{\mathord}{lettersA}{'216} 
 \DeclareMathSymbol{\RRR}{\mathord}{lettersA}{'222}%
 \DeclareMathSymbol{\ZZZ}{\mathord}{lettersA}{'232} 
 \DeclareMathSymbol{\QQQ}{\mathord}{lettersA}{'221} 
 \DeclareMathSymbol{\CCC}{\mathord}{lettersA}{'203} 

\title{
A new approach for solving mixed integer DC programs using a continuous relaxation with no integrality gap and smoothing techniques
}
\date{}
\author{Takayuki Okuno \and Yoshiko T. Ikebe%
\thanks{Department of Information and Computer Technology
Tokyo University of Science 6-3-1 Niijuku, Katsushika-ku, Tokyo 125-8585, Japan
(\{t\_okuno,yoshiko\}@rs.tus.ac.jp )}
}
\begin{document}
\maketitle
\begin{abstract}
In this paper, we consider a class of 
mixed integer programming problems (MIPs) whose objective functions are DC functions,
 that is, functions representable in terms of the difference of two convex functions.
These MIPs contain a very wide class of computationally difficult nonconvex MIPs since the DC functions have powerful expressability.
Recently, Maehara, Marumo, and Murota provided a continuous reformulation 
without integrality gaps for discrete DC programs having only integral variables. 
They also presented a new algorithm to solve the reformulated problem.  
Our aim is to extend their results to MIPs and give two specific algorithms 
to solve them.
First, we propose an algorithm based on DCA originally proposed by Pham Dinh and Le Thi, where convex MIPs are solved iteratively. 
Next, to handle nonsmooth functions efficiently, 
we incorporate a smoothing technique into the first method.  
We show that sequences generated by the two methods converge to stationary points under some mild assumptions. 
\end{abstract}
\section{Introduction}\label{sec1}
Let us consider the following optimization problem:
\begin{equation}\notag 
\min\;  f(x)\ {\rm sub. to }\; x=(x_M, x_N)\in S,\ x_M \in \real^M, \;x_N \in \integer^N.
\end{equation}
Here $f: \real^n \to \real \cup\{\infty\}$ is a closed proper function, i.e.,  
$f$ is lower semicontinuous and its effective domain ${\rm dom}\,f:=\{x\in \mathbb{R}^n\mid f(x)<\infty\}$
is not empty. Moreover,
$S \subseteq \real^n$ is a nonempty closed convex set, and $M$ and $N$ are disjoint sets of indices such that
$M \cup N = \{1,2,\ldots,n\}$, 
and $x_M = (x_i)_{i \in M},\, x_N=(x_i)_{i \in N}$.

In the case where $f$ is a linear or convex quadratic function 
and $S$ is represented with only linear or convex quadratic inequalities,
the branch-and-bound method and cutting plane techniques work nicely 
in the practical sense. 
Indeed, there are many commercial and free solvers implementing them, e.g. CPLEX \cite{cplex}, 
gurobi \cite{gurobi} and SCIP \cite{Achterberg2009}.
On the other hand, for the general nonlinear case, the above problem 
is extremely difficult to solve.
There are a number of ways to approach such 
mixed-integer nonlinear problems.
One method is to extend the framework of branch-and-bound to the continuous spaces
\cite{Floudas1999, Sherali2001, Tawarmalani2004}.
Another is to  utilize sequential quadratic programming (SQP)
\cite{Fletcher1994, Leyffer2001, Exler2007, Exler2012}.
These algorithms incorporate such techniques as 
trust regions, outer approximations and branch-and-bound techniques to
solve quadratic problems approximating the original one.
In particular, in applying these SQP-type algorithms 
to mixed-integer convex problems with continuously differentiable convex functions, 
global convergence to an optimum can be proved.
There are also algorithms which deal solely with mixed-integer nonlinear programs 
{with convex $f$} 
\cite{Geoffrion1972, Gupta1985, Duran1986, Westerlund1995, Bonami2009}.
See, for example, the surveys \cite{Bonami2012} and \cite{Burer2012}.

In this paper, we consider the case where 
$f$ is a so-called {\it DC} function, that is, 
a function representable as the difference of two convex functions:
\begin{equation}\label{prob1}
\begin{array}{ll}
\min\; & f(x)=g(x)-h(x)\\
{\rm sub. to }\; &x\in S,\ x_M \in \real^M, \;x_N \in \integer^N.
\end{array}
\end{equation}
where  
$g:\mathbb{R}^n\to \mathbb{R}\cup\{+\infty\}$   
and $h:\mathbb{R}^n\to \mathbb{R}{\cup \{+\infty\}}$ are closed proper convex functions.
{Hereafter, we suppose $\infty-\infty = \infty$ for convention, whereby $\emptyset\neq {\rm dom}\,g\subseteq {\rm dom}\,h$ and ${\rm dom}\,g={\rm dom}\,f$ naturally hold.}

The class of DC functions covers a very wide range of functions.
For example, any twice continuously differentiable function is DC, 
moreover, functions generated by applying operators such as $\sum$, $\Pi$, $|\cdot|$, 
and $\max(\cdot,\cdot)$ to DC functions also belong to the class DC 
\cite{Horst1999,horst2000introduction}.
Hence, the problem of our focus, \eqref{prob1} 
covers a wide class of mixed integer programs.
Note however, that given a DC function $f$, finding two explicit convex functions
$g$ and $h$ representing $f$ is a hard open problem. 
Among the functions for which a DC representation is easily found, perhaps
the most common are the quadratic functions.
{In this paper, we assume that one DC representation is explicitly given; 
how we obtain it will not enter our discussion.}

The DC programming in continuous variables is an important field of 
research in continuous optimization,
and theoretical and practical aspects have been extensively studied
\cite{PhamDinh1997, PhamDinh2014}. 
For example, the global optimality condition is completely characterized by the Toland-Singer duality theorem.
This duality theorem in turn forms the basis for the fundamental DC programming
algorithm known as DCA\cite{PhamDinh1997}, which is known to have nice convergence properties.

DC programming also has many useful applications.  
One example is in mixed-integer linear programs,
where integer constraints on variables are incorporated into the objective
functions via penalty functions \cite{Niu1997}. 
Other notable results have been reported in sparse optimization 
\cite{Thiao2010, LeThi2015} 
and portfolio selection \cite{LeThi2009}. 
This is an active field, with remarkable recent progress in both theory
and applications.  

On the other hand, discrete DC programming, which concerns DC programs with 
integrally constrained variables, that is, \eqref{prob1} with {$M = \emptyset$ and}
$N=\{1,2,\ldots,n\}$, is still a relatively unexplored area.
Recently, a promising approach was proposed by Maehara and Murota \cite{Maehara2015a}, 
who showed how the framework of discrete convex analysis can be applied,
to export results in continuous DC theory to a discrete setting.
This was further pursued in Maehara, Marumo and Murota \cite{Maehara2015b}, 
who proved a powerful result in constructing continuous relaxations of 
discrete DC programs. 
The simplest continuous relaxation for \eqref{prob1} 
may just replace $\integer^N$ by $\real^N$.
As is well known, this does not work effectively in general, since 
an integrality gap usually occurs, that is, the optimal values 
of the original and relaxed problems do not coincide. 
On the other hand, the new continuous relaxation proposed in \cite{Maehara2015b} 
replaces $g$ with its closed convex closure (and $h$ with an arbitrary
relaxation).
Its notable property is that no integrality gap is generated.

In this paper we 
extend the theorem of Maehara, Marumo and Murota to mixed integer DC programs 
of the form\,\eqref{prob1},
and propose two algorithms to solve them.
Our first algorithm, a generic scheme  based on the DCA originally proposed by 
Pham Dinh and Le Thi\,\cite{PhamDinh1997}, 
iteratively solves a sequence of convex mixed integer programs.
Our second algorithm, designed to deal with nonsmooth functions,
is obtained by 
incorporating smoothing techniques\, \cite{chen2012smoothing} into the first.
The sequences generated by both two methods converge to stationary points under 
some mild assumptions. 

This paper is organized as follows.
In Section~\ref{sec3} we briefly describe existing results in 
continuous and discrete DC programming, and in Section~\ref{sec2}, 
we show how to extend the theorem of Maehara, Marumo and Murota to obtain
a continuous relaxation of \eqref{prob1} with no integrality gap.
Next, in Section~\ref{sec4} we describe our basic first algorithm,
and in Section~\ref{sec5} we give our second algorithm with smoothing techniques along with results in convergence properties.

{Throughout the paper, we will use the following notations:
For any $x\in \real^n$, $\|x\|$ represents the Euclidean 2-norm of $x$.
For any nonempty set $X\subseteq \real^n$, we denote the convex hull and closure of $X$ by
${\rm co}\,X$ and ${\rm cl}\,X$, respectively.
Also, we denote the interior and relative interior of $X$ by ${\rm int}\,X$ and ${\rm ri}\,X$, respectively.
Let $\varphi: \real^n \to \real \cup \{+\infty\}$ be a convex function. 
For $x \in \dom \varphi$, the subdifferential of $\varphi$ at $x$, that is, the set of all subgradients of $\varphi$ at $x$, is denoted by $\partial \varphi(x)$.
We write the conjugate of $\varphi$ as $\varphi^{\ast}$, that is, 
the function 
$\varphi^{\ast}: \real^n \to \real \cup \{+\infty\}$ defined by 
$\varphi^{\ast}(y) = \sup_{x \in \real^n}\big\{ \langle y,x\rangle - \varphi(x)\big\}$
where $\langle y,x\rangle$ stands for the canonical inner product of $y$ and $x$, i.e., $\langle y,x \rangle = y^{\top}x$. 
For $\psi: \integer^n (\mbox{resp.,}\ \real^n) \to \real \cup \{+\infty\}$ 
the epigraph of $\psi$ is the set 
$\epir{\integer^n}\psi\ ({\rm resp., }\ \epir{\real^n} \psi):=\{(x,\,x_{n+1}) \mid x_{n+1} \geq \psi(x),\,
x \in \integer^n (\mbox{resp.,}\ \real^n) \} \subseteq \real^{n+1}$.}
Finally, $\real^n_{+(++)}$ is the non-negative (positive) orthant in $\real^n$.
\section{A brief review of continuous and discrete DC programmings}\label{sec3}
We begin by considering \eqref{prob1} with $S=\real^n$ and $N=\emptyset$, more specifically, 
\begin{equation}\label{DCP}
\min_{x \in \real^n} \{g(x)-h(x)\}. 
\end{equation}
Then, the following proposition holds.
\begin{proposition}[\cite{PhamDinh1997}]
Suppose that the DC program (\ref{DCP}) has an optimal solution $x^{\ast}$. Then, we have
\begin{enumerate}
\item 
$\partial g(x^{\ast}) \supseteq \partial h(x^{\ast})$,
\item $\bar{y} \in \partial h(x^{\ast}) \Leftrightarrow 
x^{\ast} \in \partial h^{\ast}(\bar{y})$, and
\item $\bar{y} \in \partial h(x^{\ast}) \Rightarrow $
$\bar{y}$ is an optimal solution of 
$\inf_{y \in \real^n} \{h^{\ast}(y)-g^{\ast}(y)\}$. 
\end{enumerate}
\end{proposition}
The following theorem is known as Toland-Singer duality, and forms the basis
for DC minimization algorithms.
\begin{theorem}\label{toland-singer}(Toland-Singer duality)
\[
\inf_{x \in \real^n} \{g(x)-h(x)\} 
=
\inf_{y \in \real^n} \{h^{\ast}(y)-g^{\ast}(y)\} 
\]
\end{theorem}
We next define stationary points for 
DC programs that contain global optima.
\begin{definition}\label{def-stationary}
A stationary point for $g-h$ is a point $x^{\ast}$ such that 
\[
\partial g(x^{\ast}) \cap \partial h(x^{\ast}) \neq \emptyset.
\]
\end{definition}
Let us introduce an existing algorithm for solving the DC program which will become 
the base of our proposed algorithms, and cite its convergence results. 
For details we refer the reader to \cite{PhamDinh1997}.
\begin{center}
\underline{\sc{Simplified DC Algorithm(DCA)}}
\end{center}
\begin{description}
\item[Step~0:]
Choose $x^0 \in \real^n$. Set $k=0$
\item[Step~1:]
Choose $y^{k}\in \partial h(x^k)$ and $x^{k+1} \in \partial g^{\ast}(y^{k})$
\item[Step~2:] If stopping criterion is satisfied stop, 
\item[\phantom{\quad \quad}]
else set $k=k+1$ and go to Step 1
\end{description}
\begin{theorem}[\cite{PhamDinh1997}]
Let $\{x^k\}$ and $\{y^k\}$ be the sequences generated by the simplified DCA.
Then, the following statements hold.
\begin{enumerate}
\item $g(x^{k+1})-h(x^{k+1}) \leq g(x^{k})-h(x^{k})$. 
\item 
$h^{\ast}(y^{k+1})-g^{\ast}(y^{k+1}) \leq h^{\ast}(y^{k})-g^{\ast}(y^{k})$.
\item 
Every accumulation point $x^{\ast}$ ($y^{\ast}$) of the sequence
$\{x^k\}$ ($\{y^k\}$) is a stationary point of $g-h$ ($h^{\ast}-g^{\ast}$).
\end{enumerate}
\end{theorem}

We now turn to DC programs with discrete variables.
Before introducing the results of Maehara, Marumo and Murota,
we define some concepts related to
discrete functions. 
Consider a function on discrete variables, 
$\varphi: \integer^n \to \real \cup \{+\infty\}$.
\begin{definition}\label{def-conv}
A convex function $\hat{\varphi}: \real^n \to \real \cup \{+\infty\}$ is a
\textit{convex extension} of $\varphi$ if 
\[\hat{\varphi}(x) = \varphi(x) \quad (x \in \integer^n). \]
The \textit{convex closure} of $\varphi$ is the function 
$\varphi^{\rm cl}: \real^n \to \real \cup \{+\infty\}$
whose epigraph is equal to the closed convex hull of the epigraph of $\varphi$
, i.e.,
$$
\epir{\real^n}\varphi^{\rm cl}=\clco\epir{\integer^n} \varphi.
$$
\end{definition}
While the convex closure can be defined for any $\varphi$, 
clearly, not all discrete functions have convex extensions.
If the discrete function $\varphi$ does have a convex extension 
$\hat{\varphi}$, then we always have
\[\varphi^{\rm cl}(x) = \hat{\varphi}(x) \quad (x \in \integer^n). \]
As, in this paper,  we will be concerned only with discrete functions which
are the restrictions of continuous convex functions on $\real^n$ to 
$\integer^n$, all discrete functions will trivially have 
convex extensions. 

Let us consider the DC program {\eqref{prob1} with $S=\real^n$
in which all variables are restricted to integer values,
i.e., $M = \emptyset$ and $N=\{1,2,\ldots,n\}$:
\begin{equation}\label{discreteDCP}
\min_{x \in \integer^n}\;  \big\{ g(x) - h(x) \big\}.
\end{equation}

If we define the discrete functions
 ${g}_{\small\integer}, \, h_{\integer}: \integer^n \to \real \cup \{+\infty\}$ 
as the restrictions of $g$ and $h$ to $\integer^n$: 
\begin{equation}\label{discreteDCPrelax}
{g}_{\small\integer}(x) = g(x), \; h_{\integer}(x) = h(x)  \quad (x \in \integer^n)
\end{equation}
and let $\hat{g}, \, \hat{h}: \real^n \to \real \cup \{+\infty\}$ 
be any convex extensions of ${g}_{\small\integer}$ and $h_{\integer}$,
then the following continuous DC program is clearly a relaxation of 
\eqref{discreteDCP}
\begin{equation}
\min_{x \in \real^n} \{\hat{g}(x)-\hat{h}(x)\}. \label{reDCP}
\end{equation}
The original functions $g$ and $h$ are obvious candidates for 
the convex extensions $\hat{g}$ and $\hat{h}$, but this is usually a 
poor choice  as the two optimal values 
of \eqref{discreteDCP} and \eqref{reDCP} generally do not coincide. 
Maehara, Marumo and Murota \cite{Maehara2015b} proved that the appropriate choice of  
$\hat{g}$ ensures this will not happen.
\begin{theorem}[\cite{Maehara2015b}]\label{th-mmm}
If $\hat{g}$ is the convex closure of ${g}_{\small\integer}$, then 
the optimal values of the two problems \eqref{discreteDCP} and 
\eqref{reDCP} coincide.
\end{theorem}

We now turn to our main concern,  mixed integer DC programs.
\section{Continuous relaxation with no integrality gap}\label{sec2}

We begin by rephrasing problem \eqref{prob1}.
By using the indicator function of set $S$, that is, 
the function $\delta_S: \real^n \to \real \cup \{+\infty\}$ defined by
\[
\delta_S(x) = \left\{\begin{array}{ll}
0 & (x \in S) \\
+\infty & (x \not \in S)
\end{array}\right. ,
\]
\eqref{prob1} can be written as  
\begin{equation}\label{prob2-2}
\begin{array}{ll}
\min & \big(\delta_S(x) + g(x) \big) - h(x)\\
{\rm sub. to } & x=(x_M,x_N)\in \real^M\times \integer^N.
\end{array}
\end{equation}
Since $S$ is a closed convex set, $\delta_S$, and hence $\delta_S + g$ 
are closed proper convex functions.

Now define $\tilde{g}$ and $\tilde{h}$ respectively as the restrictions of 
$\delta_S + g$ and $h$ to $\real^M\times \integer^N$, that is,
\begin{equation*}
\tilde{g}:=\left(\delta_S + g\right)|_{\real^M\times \integer^N}\ \mbox{and }\tilde{h}:=h|_{\real^M\times \integer^N}.
\end{equation*}
We also denote the convex closure of $\tilde{g}$ by $\tilde{g}^{\rm cl}$.
Convex extensions, epigraphs, and convex closures of $\tilde{g}$ and 
$\tilde{h}$ are defined in a manner analogous to the discrete functions
in the last paragraph of Section~\ref{sec1} and Definition~\ref{def-conv};
for example, the epigraph of $\tilde{g}$ is defined as the set 
$
\{(x_M, x_N, x_{n+1}) \in \real^M \times \integer^N \times \real 
\mid x_{n+1} \geq g(x_M, x_N) \}=:\epir{\real^M\times \integer^N}\tilde{g}$.

In the rest of this section, we extend the theorem of Maehara, Marumo and Murota, to 
mixed-integer DC programs \eqref{prob1}, that is, DC programs  
involving both integer-valued and continuous variables. More precisely, we prove the following theorem.
\begin{theorem}\label{thm:1023}
Let $\hat{h}:\real^n\to \real\cup\{\infty\}$ be an arbitrary convex extension of $\tilde{h}$.
Then, the following (continuous) DC program:
\begin{equation}\label{prob30}
\min_{x \in \real^n} \big\{\tilde{g}^{\rm cl}(x) - \hat{h}(x)\big\}
\end{equation}
has the same optimal value as the mixed-integer DC program \eqref{prob2-2}, i.e., as \eqref{prob1}.
In particular, the optimal set of \eqref{prob2-2} is contained in 
that of \eqref{prob30}.
\end{theorem}
\begin{proof}
If \eqref{prob1} is not bounded from below, then neither is \eqref{prob30} since 
\eqref{prob30} is a relaxation of \eqref{prob1}.
Hence, we only consider the case where \eqref{prob1}, i.e, \eqref{prob2-2} 
has a finite optimal value. 
Let $\tilde{v^{\ast}}$ and $v^{\ast}$ respectively be the optimal values of problems
\eqref{prob2-2} and \eqref{prob30}. 
We will prove that 
$\tilde{v^{\ast}} = v^{\ast}$.
Since \eqref{prob30} is a relaxation of \eqref{prob1}, it suffices to show $\tilde{v^{\ast}} \le v^{\ast}$.
By optimality of $\tilde{v^{\ast}}$ for \eqref{prob2-2},
\[
\tilde{g}(x_M,\,x_N)  \geq 
\tilde{h}(x_M,\,x_N) + \tilde{v^{\ast}} \qquad 
(\;(x_M,x_N) \in \real^M \times \integer^N) 
\]
implying that 
\begin{equation}
\emptyset \neq \epir{\real^M\times \integer^N} \tilde{g} \subseteq \epir{\real^M\times \integer^N} \tilde{h} + (0,\,\tilde{v^{\ast}})
\subseteq \epir{\real^n} \hat{h} + (0,\,\tilde{v^{\ast}}), \label{eq:1220}
\end{equation}
where the leftmost 
inequality
is from the assumption that ${\rm dom}\,g\neq \emptyset$.
Notice that from \eqref{eq:1220} and $\epir{\real^M\times\integer^N}\tilde{g}\subseteq \epir{\real^n}\tilde{g}^{\rm cl}$, we obtain
$
D:=\epir{\real^n} \tilde{g}^{\rm cl} \cap \left(\epir{\real^n} \hat{h} + (0,\,\tilde{v^{\ast}})\right)\neq \emptyset.
$
Now, we show that  
\begin{equation}
\epir{\real^n} \tilde{g}^{\rm cl} \subseteq \big(\epir{\real^n} \hat{h} + (0,\,\tilde{v^{\ast}})\big).\label{eq:1220-2}
\end{equation}
To this end, suppose that 
$\epir{\real^n} \tilde{g}^{\rm cl} \setminus 
\big(\epir{\real^n} \hat{h} + (0,\,\tilde{v^{\ast}})\big) \neq \emptyset$ for contradiction.
Then $D$ is a nonempty closed convex subset such that 
\begin{equation}
\epir{\real^n} \tilde{g}^{\rm cl}\supsetneq D \supseteq \epir{\real^M\times \integer^N} \tilde{g}.\label{eq:1220-3}
\end{equation}
Moreover, there must be some closed proper convex function 
$\phi:\real^n \to \real \cup \{+\infty\}$ satisfying $\epir{\real^n} \phi = D$.
Obviously $\phi$ is a closed convex extension of $\tilde{g}$  
from $\real^M\times \integer^N$
to $\real^n$.
However, this contradicts the fact that $\tilde{g}^{\rm cl}$ is the convex closure of $\tilde{g}$,
because $\epi \phi$ is a proper subset of 
$\epi \tilde{g}^{\rm cl}$ by \eqref{eq:1220-3}.
Therefore, we have 
\eqref{eq:1220-2}
which further yields $\tilde{v}^{\ast}\le v^{\ast}$.   
This completes the proof.
\end{proof}

By the above result, it is justified to solve \eqref{prob30} instead of \eqref{prob1}.
In the remainder of the paper, 
we propose a specific algorithm for solving \eqref{prob30}.
In our algorithm, we choose $h$ as $\hat{h}$, a convex extension of $\tilde{h}$.
Therefore, our target is to solve the following problem 
\begin{equation}\label{prob3}
\min_{x \in \real^n} \big\{\tilde{g}^{\rm cl}(x) - h(x)\big\}.
\end{equation}

\section{A basic algorithm for the mixed integer DC program}\label{sec4}

In this section, we formulate a basic algorithm 
based on the DCA of Section~\ref{sec3}, for solving \eqref{prob3}.
For the DC program \eqref{prob3}, 
recall that the DCA involves finding $x_k$ and $y_k$ with 
\[
y^{k} \in \partial h(x^k),  \mbox{ and }
x^{k+1} \in \partial (\tgcl)^{\ast}(y^{k}).
\]
Finding $x^{k+1} \in \partial (\tgcl)^{\ast}(y^{k})$ 
can be accomplished by using the following relations:
\begin{eqnarray*}
 x^{k+1} \in \partial (\tgcl)^{\ast}(y^{k}) 
& \Leftrightarrow & \partial \tgcl(x^{k+1}) \ni y^{k} \\
& \Leftrightarrow 
& x^{k+1} \mbox{ is a solution of }
\inf_{w \in \real^n} \big(\tgcl(w)-\langle y^{k},\,w \rangle \big)
\end{eqnarray*}
The rightmost optimization problem involves minimizing a convex function.
However, this cannot be solved by using standard convex optimization 
methodologies such as the interior point method,
 since we do not have an explicit expression of $\tgcl$ in general.
This can be overcome by using Theorem~\ref{thm:1023} to note that it
corresponds to solving the following convex mixed integer program:
\begin{equation}
\begin{array}{ll}
\min & g(x) - \langle y^{k},\,x\rangle \\
{\rm sub. to } & x \in S,\ x=(x_M,x_N) \in \real^M\times \integer^N.
\end{array}\label{amuro}
\end{equation}
Hence, by replacing 
$x^{k+1}\in\partial (\tgcl)^{\ast}(y^{k})$ with \eqref{amuro} in Step~2 of the simplified DCA,
we gain a specific algorithm for solving \eqref{prob1} as below: 
\begin{center}
{\sc{\underline{Sequential convex mixed-integer}\\ \underline{programming method (SCMIP)}}}
\end{center}
\begin{description}
\item[Step~0:]
Choose $x^0 \in \real^n$. Set $k=0$.
\item[Step~1:]
Choose $y^{k} \in \partial h(x^k)$ and solve \eqref{amuro} to obtain $x^{k+1}$.
\item[Step~2:]
If stopping criterion is satisfied, stop, 
\item[\phantom{\quad \quad}]
else set $k=k+1$ and go to step 1
\end{description}
Obviously, each iteration point $x^k$ is feasible to \eqref{prob1}. 
By applying existing results on convergence for the DCA\,\cite{PhamDinh1997,PhamDinh2014}, 
we can make some observations for the case that at least one of  
$\tgcl$ and $h$ 
is a strongly convex function.
\begin{itemize}
\item 
both $\tgcl(x^{k})-h(x^{k})(=f(x^k))$ and $h^{\ast}(y^{k})-(\tgcl)^{\ast}(y^{k})$ 
strictly decrease
\item 
if $x^{\ast}$(resp., $y^{\ast}$) is an accumulation point of $\{x^{k}\}$ (resp., $\{y^k\}$),
then $x^{\ast}$(resp.,$y^{\ast}$) is a stationary point of $\min \tgcl(x)-h(x)$ (resp., $\min h^{\ast}(y)-(\tgcl)^{\ast}(y)$). That is to say, $y^{\ast}\in \partial\tgcl(x^{\ast})\cap \partial h(x^{\ast})$ and $x^{\ast}\in \partial(\tgcl)^{\ast}(y^{\ast})\cap \partial h^{\ast}(y^{\ast})$ hold.
\item 
the $x^N$-part of $x^k$ converges to some integer point within finitely many iterations.%
\footnote{This property is specific to the case where the DCA is applied to \eqref{prob3}. For the proof, see the convergence analysis of the next algorithm.}
\end{itemize}
In the above discussion, the assumption that at least one of 
$\tgcl$ and $h$ is strongly convex is crucial.
We place emphasis on the ``at least one'' phrase.

In problem \eqref{prob1}, we did not assume strong convexity of either $g$ or $h$.
Thus, at first glance, the above results may seem inapplicable, however, 
it can be easily overcome 
by considering the following equivalent problem for fixed $\rho>0$:
\begin{equation}
\min\ \big(g(x)+\rho\frac{\|x\|^2}{2}\big)
 - \big(h(x)+\rho\frac{\|x\|^2}{2}\big)\ \mbox{sub.to}\ x\in S,\ (x_M,x_N) \in \real^M\times \integer^N.\label{eq:1024_1}
\end{equation}
We note here that the convex closure of $g(\cdot)+\rho\|\cdot\|^2/2$ is usually not strongly convex,
whereas $h(\cdot) + \rho \|\cdot\|^2/2$ always is.
Thus it is important that we do not need the strong convexity of both 
$g$ and $h$.

Before ending 
this section, we make an important remark 
concerning the drawbacks of
transforming \eqref{prob1} to \eqref{eq:1024_1}.
Consider two different DC-decompositions $(g_1,h_1)$ and $(g_2,h_2)$ for $f$, i.e., $f=g_1-h_1=g_2-h_2$,
and corresponding continuous DC programs of the form\,\eqref{prob3}.
Their two optimal sets are exactly the same.
However, their sets of stationary points 
may possibly differ.
This phenomenon does not occur in continuous DC programs 
without discrete variables, and thus 
it is characteristic of \eqref{prob3}. 
To illustrate it, 
let us consider the following trivial mixed integer program: 
\begin{example}
\begin{equation}
{\rm min}\hspace{0.5em}x\hspace{0.5em}\mbox{\rm sub.~to }x\in \{-1,0,1\} \label{ex0}
\end{equation}
\end{example}
Choose two DC-decompositions $(g_1, h_1) = (x,0)$ and 
$(g_2, h_2) = (x^2+x,x^2)$. 
Then, $\tilde{g}_1^{\rm cl}(x) = x$, $\tilde{g}_2^{\rm cl}$
is the polygonal line connecting the three points $(-1,0), (0,0)$ and $(1,2)$, and $\dom \tilde{g}_1^{\rm cl} = \dom \tilde{g}_2^{\rm cl} = [-1,\,1]$.
Thus the resulting optimization problems of the form \eqref{prob3} are:
\[
\min \left\{\begin{array}{ll}
\infty \quad & (x<-1) \\ 
x & (-1 \leq x \leq 1) \\
\infty & (1 < x)
\end{array}\right.
\qquad \mbox{and} \qquad 
\min \left\{\begin{array}{ll}
\infty & (x<-1) \\ 
-x^2 & (-1 \leq x \leq 0) \\
2x-x^2 \quad & (0 \leq x \leq 1) \\
\infty & (1 < x)
\end{array}\right.
\]
%
%
The set of  stationary points of the former problem is nothing but 
the optimal set $\{-1\}$ of \eqref{ex0}, while that of the latter is $\{0,-1\}$.
This example indicates that
the choice of DC decomposition may affect efficiency in finding the optima.
\section{SCMIP with smoothing techniques}\label{sec5}
In this section, we focus on a particular case of \eqref{prob1} where either one of $g$ and $h$ is not differentiable on the effective domain.
Such problems often  occur in applications.
For example, consider problem\,\eqref{prob1} equipped with a dc inequality 
constraint $g_1(x)-g_2(x)\le 0$ with $g_1,g_2:\real^n\to \real$ being convex functions.
If $g_1-g_2$ is not convex, 
it is difficult to directly apply the SCMIP.  
One remedy for this is to lift the constraint $g_1(x)-g_2(x)\le 0$ into the objective function 
as a penalty term $\tau \max(g_1-g_2,0)$ with $\tau>0$ being a penalty parameter, 
then further decomposing it as follows:
\begin{equation}\notag
\begin{array}{ll}
\min\; & \left(g(x)+\tau\max\left(g_1(x),g_2(x)\right)\right)-\left(h(x)+\tau g_2(x)\right)\\
{\rm sub. to }\; &x\in S,\ x=(x_M, x_N) \in \real^M \times \integer^N.
\end{array}
\end{equation} 
This problem has the form \eqref{prob1},
 since $g+\tau\max(g_1,g_2)$ and $h+\tau g_2$ are both convex.
Obviously, $g+\tau\max(g_1,g_2)$ is not differentiable in general,
due to the existence of the max function. 
Theoretically, the SCMIP is applicable
regardless of the differentiability of $g$ and $h$.
However, in this case we must iteratively solve mixed-integer nonsmooth optimization
problems. 
Practically, nonsmooth problems are not as tractable as smooth ones,
even if they posses only continuous variables.
Moreover, most available free or commercial solvers cannot deal with them.
Thus, we  employ the smoothing method, which is one of the most powerful techniques 
for solving optimization problems or nonlinear equations involving 
nondifferentiable functions.
This method solves a sequence of approximated problems in which 
given nonsmooth functions are replaced by so-called smoothing functions. 
For a comprehensive survey on the smoothing method, refer to \cite{chen2012smoothing} and references therein.
In this paper, the smoothing functions are defined as below.
\begin{definition}\label{def:smoothing}
Let $\phi:\real^n\to \real\cup\{+\infty\}$ be 
a function such that ${{\rm int}}\,\dom\phi\neq\emptyset$ and
$\phi$ is continuous on $\dom \phi$.
We say that 
$\Phi:\real^n\times \real_+\to \real\cup\{+\infty\}$ is a smoothing function of $\phi$ when 
\begin{enumerate}
\item[$(1)$] $\dom\Phi(\cdot,\mu)=\dom \phi$ for any $\mu>0$,
\item[$(2)$] $\Phi(\cdot,\mu)$ is continuously differentiable on ${\rm int}\,\dom\Phi$ for any $\mu>0$, 
\item[$(3)$]$\displaystyle \lim_{z\in \dom\phi\to x,\,\mu\to +0}\Phi(z,\mu)=\phi(x)$ holds for any $x\in \dom\phi$, and 
\item[$(4)$]$\Phi(\cdot,0)=\phi(\cdot)$.
\end{enumerate}
\end{definition}
Various kinds of smoothing functions for specialized problems 
have been studied extensively. 
For example, consider the plus function $(\cdot)_+:=\max(\cdot,0)$. 
Many nonsmooth functions such as $|x|$, $\max(x,y)$, and $\min(x,y)$ 
can be explicitly represented with the plus function,
thus it is very versatile.
One way of approximating the plus function is by using a piecewise continuous function
$p:\real\to \real_+$ such that $\int_{-\infty}^{+\infty} p(s)ds=1$, 
$
p(s)=p(-s),\ \mbox{and }\int_{-\infty}^{+\infty}|s|p(s)ds<\infty.
$
For such $p$, it is well-known that 
$\Phi(t,\mu):=\int_{-\infty}^{+\infty}\{(t-\mu s)_+p(s)\}ds$ 
%
becomes a smoothing function for $(\cdot)_+$. 
This class of smoothing function is called the Chen-Mangasarian function\,\cite{chen1996class} and admits the following properties:
\begin{description}
\item[$(P1)$] For any $\mu\ge 0$, $\Phi(\cdot,\mu)$ is convex.
\item[$(P2)$] 
For any $x\in\dom\phi$, $\Phi(x,\cdot)$ is a nondecreasing function on $\real_+$ and furthermore there exists some $\kappa>0$ such that 
$0\le
\Phi(x,\mu_2)-\Phi(x,\mu_1) 
\le \kappa(\mu_2-\mu_1)$ holds for any $x\in \real^n$ and $0<\mu_1\le \mu_2$.
\\
\end{description}
Properties $(P1)$ and $(P2)$ often play a crucial role in establishing convergence analysis 
for the smoothing method.
They are satisfied by many existing smoothing functions for $(x)_+$, 
$\max(x,y)$, $\min(x,y)$ and $|x|$. 
In the rest of this section, we will establish
convergence properties under $(P1)$ and $(P2)$.

Now, let us turn back to \eqref{prob1} and let
$G,H:\real^n\times \real_+\to \real$ 
respectively
be smoothing functions for $g$ and $h$.
Hereafter, 
to simplify notations,
for any $u\in \real^2_+$ and $\{u^k\}_{k\ge 0}\subseteq \real^2_+$, we often write  
$$
u=(\mu_1,\mu_2),\hspace{1.0em}u^k=(\mu^k_1,\mu^k_2).
$$  
In addition, we denote 
$$G_{\mu}(\cdot)=G(\cdot,\mu),\hspace{1.0em}H_{\mu}(\cdot)=H(\cdot,\mu).$$ 
Following the terminology of Section~\ref{sec2},
let
$$
\tilde{G}_{\mu}=\left(G_{\mu}+\delta_S\right)|_{\real^M\times \integer^N},
$$
$\tilde{G}^{\rm cl}_{\mu}$ be 
the convex closure of $\tilde{G}_{\mu}$, and then
define $f_{u},\tilde{f}^{\rm cl}_{u}:\real^{n}\to \real\cup \{+\infty\}$ by 
$$
f_{u}:=G_{\mu_1}-H_{\mu_2},\hspace{1.0em}\tilde{f}^{\rm cl}_{u}:=\tilde{G}^{\rm cl}_{\mu_1}-H_{\mu_2},$$ respectively. 

We now
propose a new algorithm incorporating the smoothing method into the SCMIP presented in the previous section.
In our method, 
instead of solving \eqref{amuro},
we execute the SCMIP solving the following mixed integer convex 
optimization problem with a smoothing parameter $u^k\in \real^2_{+}$ such that $\lim_{k\to \infty} u^k=0$: 
\begin{equation}
 \begin{array}{rll}
     &\mbox{min}& 
G_{\mu^{k}_1}(x)- \langle y^k,x-x^k\rangle\\
     &\mbox{sub.to } &x\in S,\ x=
(x_M,x_N) \in \real^M\times \integer^N,                   
 \end{array}\label{bright2}
\end{equation}
where $y^k\in\partial H_{\mu^k_2}(x^k)$. 
This is equivalent to solving the continuous convex program: 
\begin{equation}
\mbox{min}\ \tilde{G}^{\rm cl}_{\mu^{k}_1}(x)- \langle y^k,x-x^k\rangle. \label{bright3}
\end{equation}
The overall 
framework of the smoothing SCMIP is described as follows:
\begin{center}
\underline{\sc{Smoothing SCMIP}}
\end{center}
\begin{description}
\item[Step~0:]
Choose $x^0 \in \real^n$, $u^0=(\mu^0_1,\mu^0_2)\in\real_+^2$, and $\gamma\in (0,1)$.
Set $k=0$
\item[Step~1:]
Set ${y^{k}}\in\partial H_{\mu^k_2}(x^k)$ ($y^{k}=\nabla H_{\mu^k_2}(x^k)$ if $\mu^k_2>0$) and solve \eqref{bright2}
to obtain $x^{k+1}$.
\item[Step~2:]
If stopping criterion is satisfied stop, 
\item[\phantom{\quad \quad}]
else set 
$u^{k+1}=\gamma u^k$, and $k=k+1$ and go to Step 1
\end{description}
\begin{rem}
Note that in Step~0 we do not restrict the scope of $\{u^k\}$ to the positive orthant $\real^2_{++}$ for general versatility. 
If we set $u^k=0$ for all $k\ge 0$, the above algorithm is nothing but the SCMIP without the smoothing technique.
Also, if we set $\mu_2^k=0$ for all $k\ge 0$, the smoothing is applied 
only to the function $g$, leaving $h$ as it is. 
\end{rem}
\subsection*{Convergence analysis for the smoothing SCMIP}
In this section, we show that generated sequences $\{x^k\}$ and $\{y^k\}$ respectively converge 
to stationary points of the DC problem \eqref{prob3} and its Toland-Singer dual 
in the sense of Definition~\ref{def-stationary} replacing $g$ with $\tilde{g}^{\rm cl}$.  

Throughout our convergence analysis, we suppose that 
function $g$ has an effective domain with nonempty interior and $\dom h$ is the full space for simplicity of expression.
Moreover, we assume that the smoothing functions $G$ and $H$ are chosen so that 
they satisfy properties $(P1)$ and $(P2)$.
That is to say, we suppose that the following Assumption~$(A0)$ holds in the subsequent analysis:
{\begin{description}
\item[$(A0)$]
$
{\rm int}\,\dom g\neq \emptyset,\ \dom h=\real^n,\ \mbox{and $G$ and $H$ satisfy }(P1)\mbox{ and }(P2).
$
\end{description}}
Furthermore, we make the following assumptions:
\begin{description}
\item[$(A1)$] The set $\Omega_0:=\left\{x\in \real^n\mid f(x)\le f(x^1)+\kappa(\mu^1_1+\mu^1_2)\right\}\cap  \mathcal{F}$ is nonempty and compact, where $\mathcal{F}$ denotes the feasible region of \eqref{prob1}, i.e.,
$\mathcal{F}=S\cap\left(\real^M\times\integer^N\right)$ 
and $\kappa$ is a positive constant prescribed in $(P2)$.
\item[$(A2)$] 
For any $\mu\ge 0$,
$H_{\mu}$ is strongly convex with a modulus $\tau>0$, i.e., 
 \begin{align*}
 \alpha H_{\mu}(x)+(1-\alpha)H_{\mu}(z)-H_{\mu}(\alpha x + (1-\alpha) z)\ge \alpha(1-\alpha)\tau\|x-z\|^2/2
 \end{align*}
holds for any $x,z\in \real^n$ and $\alpha\in [0,1]$.
\item[$(A3)$] $\dom f\cap \mathcal{F}\neq \emptyset$.
\end{description}
Notice the following:
Assumption~$(A1)$ holds when $f$ is coercive on $\mathcal{F}$, 
that is, $f(x^{\ell})\to\infty$ as $\ell\to \infty$
for any sequence $\{x^{\ell}\}$ such that $\|x^{\ell}\|\to\infty$ and $x^{\ell}\in \mathcal{F}$ for all $\ell\ge 1$.
Assumption~$(A2)$ can be 
satisfied
by considering an alternative DC-decomposition
$f(\cdot) = (g(\cdot) + \tau \|\cdot\|^2/2)- (h(\cdot)+ \tau \|\cdot\|^2/2)$ 
and smoothing the function $h$, 
although the set of stationary points may change 
(see Example~1 in Section~\ref{sec4}). 
By $(A3)$, we require that \eqref{prob1} has at least one feasible point taking a finite value.

We begin by proving finiteness of the optimal value of the problem of interest.
\begin{proposition}\label{kesukamo}
Suppose that Assumptions\,$(A1)$ and $(A3)$ hold.
Then, the optimal value of \eqref{prob1}, i.e, \eqref{prob3} is finite.
\end{proposition}
\begin{proof}
We first show that $\inf_{x\in \mathcal{F}} f(x)>-\infty$.
Suppose that the contrary is true.
Then, there exists some sequence $\{x^{l}\}\subseteq \Omega_0$ such that
$
\lim_{l\to \infty}f(x^l)=-\infty.
$
where $\Omega_0$ is the nonempty and compact set defined in Assumption~$(A1)$.
Therefore, from Assumption\,$(A1)$, $\{x^l\}$ is bounded and has an accumulation point in $\Omega_0$, say $x^{\ast}\in \Omega_0$, and 
we can assume $\lim_{l\to \infty}x^l=x^{\ast}$ without loss of generality.
As $f$ is lower semicontinuous, $f(x^{\ast})\le \lim_{k\to \infty}f(x^k)=-\infty$.
However, this is a contradiction.
Therefore,  $\inf_{x\in \mathcal{F}} f(x)>-\infty$.
We also obtain $\inf_{x\in \mathcal{F}} f(x)<\infty$ in view of Assumption~$(A3)$. Hence, we have the desired result.
\end{proof}
{\begin{lemma}\label{lem:1222}
For any $x\in \real^n$ and $u=(\mu_1,\mu_2)\in\real^2_{+}$, 
the following relations hold:
\begin{equation}
g(x)\le G_{\mu_1}(x)\le g(x)+\kappa\mu_1,\ h(x)\le H_{\mu_2}(x)\le h(x)+\kappa\mu_2,\label{eq:1222}
\end{equation}
in addition, 
\begin{equation}
f(x)-\kappa\mu_2\le f_{u}(x)\le f(x)+\kappa\mu_1. \label{eq:1222-2}
\end{equation}
\end{lemma}
\begin{proof}
We first prove \eqref{eq:1222}.
Choose $x\in \real^n$ and $u\in\real^2_{+}$ arbitrarily.
Let us consider the case where $x\in \dom f=\dom g$.
From $(P2)$, we obtain 
$G_{\mu}(x)\le G_{\mu_1}(x)\le G_{\mu}(x)+\kappa(\mu_1-\mu)$.
Using these facts and $\lim_{\mu\to +0}G_{\mu}(x)=g(x)$, and 
$G_{0}(x)=g(x)$ by Definition\,\ref{def:smoothing},
we have $g(x)\le G_{\mu_1}(x)\le g(x)+\kappa\mu_1$.
When $x\notin \dom f=\dom g$, by $\dom G_{\mu}=\dom g$ from the definition of $G_{\mu_1}$, 
it follows that $g(x)=G_{\mu_1}(x)=\infty$. Thus, $g(x)\le G_{\mu_1}(x)\le g(x)+\kappa\mu_1$ is obvious. 
Similarly, we have $h(x)\le H_{\mu_2}(x)\le h(x)+\kappa\mu_2$.
Using \eqref{eq:1222} together with 
$f=g-h$ and $f_{u}=G_{\mu_1}-H_{\mu_2}$,
we readily obtain \eqref{eq:1222-2}.
\end{proof}
}
We next state some technical lemmas for establishing convergence properties of the proposed algorithm.
The following lemma assures that the 
function $\tilde{G}_{\mu}^{\rm cl}$ is well behaved. 
\begin{lemma}\label{lem:1021-2}
For any $\mu\ge 0$, the following statements hold. 
\begin{enumerate}
\item[$({\rm i})$] $\tilde{g}^{\rm cl}(x)\le 
\tilde{G}^{\rm cl}_{\mu}(x)\le 
\kappa\mu+\tilde{g}^{\rm cl}(x)$ for $x\in \dom g\cap S$
and 
\item[$({\rm ii})$] $\dom g\cap S\supseteq\dom \tilde{g}^{\rm cl}=\dom \tilde{G}^{\rm cl}_{\mu}$.
\item[$({\rm iii})$] 
Under Assumption~$(A3)$, 
${\rm ri}\,\dom \tilde{G}^{\rm cl}_{\mu}\neq \emptyset$.
\end{enumerate}
\end{lemma}
\begin{proof}
\begin{enumerate}
\item[(i)] 
From Lemma\,\ref{lem:1222},
$
g(x)\le G_{\mu}(x)\le g(x)+\kappa\mu$ for any 
$x\in\dom g\cap S$ and $\mu\ge 0$.
Thus we have $\epir{\real^M\times \integer^N}\tilde{g} 
\supseteq \epir{\real^M\times \integer^N} \tilde{G}_{\mu}
\supseteq \epir{\real^M\times \integer^N} \tilde{g}+(0,\kappa\mu)$.
Hence,
$\clco\epir{\real^M\times\integer^N} \tilde{g} 
\supseteq\clco\epir{\real^M\times \integer^N}\tilde{G}_{\mu}
\supseteq \clco\epir{\real^M\times \integer^N} \tilde{g}+(0,\kappa\mu)$ follows.
We then obtain
$
\epir{\real^n} \tilde{g}^{\rm cl}\supseteq \epir{\real^n} \tilde{G}^{\rm cl}_{\mu}
\supseteq \epir{\real^n} \tilde{g}^{\rm cl}+(0,\kappa\mu)
$ implying
$
\tilde{g}^{\rm cl}(x)
\le 
\tilde{G}^{\rm cl}_{\mu}(x)\le 
\kappa\mu+\tilde{g}^{\rm cl}(x)$.
\item[(ii)] 
Choose $x\in \dom \tilde{g}^{\rm cl}$ arbitrarily.
By the definition of $\tilde{g}^{\rm cl}$, $x\in S$ follows.
In addition, if $x\notin\dom g$, $\tilde{g}^{\rm cl}(x)=g(x)=+\infty$, which contradicts $x\in \dom\tilde{g}^{\rm cl}$ and thus $x\in\dom g$ holds.
Therefore, we obtain the first inclusion.
The second equality is readily obtained from (i).
\item[(iii)]
We can easily derive $\dom \tgcl\neq \emptyset$ from 
$\dom g \cap \mathcal{F}\neq \emptyset$ by $(A3)$.
From the nonemptiness and convexity of $\dom \tgcl$, 
we obtain ${\rm ri}\,\dom\tilde{g}^{\rm cl}\neq \emptyset$.
Then, (ii) yields  ${\rm ri}\,\dom \tilde{G}^{\rm cl}_{\mu}\neq \emptyset$.
\end{enumerate}
\end{proof}
We now
give one proposition and one lemma concerning accumulation points of values and subgradients of $\tilde{G}^{\rm cl}_{\mu}$ and $H_{\mu}$.
\begin{lemma}\label{lem:1012}
Let $x^{\ast}\in \dom \tilde{g}^{\rm cl}$
and consider sequences $\{x^{\ell}\}\subseteq \dom \tilde{g}^{\rm cl}$ and 
$\{u^{\ell}\}\subseteq \real^2_{+}$ such that $x^{\ell}\to x^{\ast}$ and $u^{\ell}\to 0$ as 
$\ell$ tends to $\infty$. Then,
$
{\displaystyle \lim_{\ell\to \infty}\tilde{G}_{\mu^{\ell}_1}^{\rm cl}(x^{\ell})=\tilde{g}^{\rm cl}(x^{\ast})}
$
holds.
\end{lemma}
\begin{proof}
From Lemma\,\ref{lem:1021-2}(ii) and $\{x^{\ell}\}\subseteq \dom \tilde{g}^{\rm cl}$, 
$\{x^{\ell}\}\subseteq \dom g\cap\mathcal{F}$ follows.
Then, by using Lemma\,\ref{lem:1021-2}(i), we have
\begin{align*}
\tilde{g}^{\rm cl}(x^{\ell})
&\le 
\tilde{G}^{\rm cl}_{\mu^{\ell}_1}(x^{\ell})\le 
\kappa\mu^{\ell}_1+\tilde{g}^{\rm cl}(x^{\ell}).
\end{align*}
Letting $\ell\to \infty$ here leads to  
\begin{align*}
\tilde{g}^{\rm cl}(x^{\ast})&\le 
\lim_{\ell\to \infty}\tilde{G}^{\rm cl}_{\mu^{\ell}_1}(x^{\ell})\le 
\tilde{g}^{\rm cl}(x^{\ast}),
\end{align*}
where we use continuity of $\tilde{g}^{\rm cl}$ on $\dom\tilde{g}^{\rm cl}$ 
and $\lim_{\ell\to \infty}x^{\ell}=x^{\ast}\in \dom \tilde{g}^{\rm cl}$. 
Hence $\lim_{\ell\to \infty}\tilde{G}^{\rm cl}_{\mu^{\ell}_1}(x^{\ell})=\tilde{g}^{\rm cl}(x^{\ast})$ holds.
\end{proof}
\begin{lemma}\label{prop:1014}
Let $x^{\ast}\in \dom \tilde{g}^{\rm cl}$, $v^{\ast}\in \real^n$, and 
$w^{\ast}\in\real^n$.
Consider sequences $\{x^{\ell}\}\subseteq \dom \tilde{g}^{\rm cl}$, $\{u^{\ell}\}\subseteq \real^2_{+}$,
$\{v^{\ell}\}\subseteq \real^n$ and $\{w^{\ell}\}\subseteq \real^n$ 
satisfying the next conditions:
\begin{itemize}
\item 
$v^{\ell}\in \partial \tilde{G}^{\rm cl}_{\mu^{\ell}_1}(x^{\ell})$ and 
$w^{\ell}\in \partial H_{\mu^{\ell}_2}(x^{\ell})$
for any $\ell\ge 1$, and   
\item 
$
\lim_{\ell\to \infty}\left(x^{\ell},u^{\ell},v^{\ell},w^{\ell}\right)=\left(x^{\ast},0,v^{\ast},w^{\ast}\right).
$
\end{itemize}
Then, $v^{\ast}\in \partial\tilde{g}^{\rm cl}(x^{\ast})$ and $w^{\ast}\in \partial h(x^{\ast})$.
\end{lemma}
\begin{proof}
It suffices to show 
\begin{align}
&\tilde{g}^{\rm cl}(x^{\ast}+z)-\tilde{g}^{\rm cl}(x^{\ast})\ge \langle v^{\ast}, z\rangle \label{eq:1019}\\
&h(x^{\ast}+z)-h(x^{\ast})\ge  \langle w^{\ast}, z\rangle\label{al:1026}
\end{align}
for any $z$. 
Since 
\eqref{al:1026} can be shown similarly to \eqref{eq:1019}, we show only \eqref{eq:1019}.
Choose $z\in \real^{n}$ arbitrarily. If $x^{\ast}+z\notin \dom\tilde{g}^{\rm cl}$, 
\eqref{eq:1019} is true since $\tilde{g}^{\rm cl}(x^{\ast}+z)=\infty$ 
while $\tilde{g}^{\rm cl}(x^{\ast})<\infty$.
Hence, we consider the case where $x^{\ast}+z\in \dom\tilde{g}^{\rm cl}$ and let $z^{\ell}:=x^{\ast}+z-x^{\ell}$ for $\ell$. Then, $x^{\ell}+z^{\ell}=x^{\ast}+z\in \dom\tilde{g}^{\rm cl}$ for any $\ell$.
Since $\tilde{G}^{\rm cl}_{\mu^{\ell}_1}$ is convex and 
$v^{\ell}\in \partial \tilde{G}^{\rm cl}_{\mu^{\ell}_1}(x^{\ell})$, we have 
\begin{equation}
\tilde{G}^{\rm cl}_{\mu^{\ell}_1}(x^{\ell}+z^{\ell})-\tilde{G}^{\rm cl}_{\mu^{\ell}_1}(x^{\ell})\ge \langle v^{\ell},z^{\ell}\rangle\label{eq:1007}
\end{equation}
for any $\ell$. Letting $\ell$ tend to $\infty$ in \eqref{eq:1007} and using Lemma\,\ref{lem:1012} imply \eqref{eq:1019}.
The proof is complete.
\end{proof}
Now, with help of the above lemmas we next prove a proposition concerning well-definedess of the proposed algorithm.
\begin{proposition}\label{prop:5-1}
Under Assumptions 
$(A1)$ and $(A2)$, the smoothing SCMIP is well-defined in the sense that $\partial H_{\mu_2^k}(x^k)$ is nonempty and
\begin{equation}
{\rm min}\ G_{\mu^k_1}(x)-\langle y ^k,x-x^k\rangle \ \mbox{{\rm sub.~to} }x\in\mathcal{F}\label{subp1}
\end{equation}
has at least one optimum for any $k\ge 0$.
Furthermore, under Assumption~$(A3)$, the optimal value is finite.
\end{proposition}
\begin{proof}
The nonemptiness of $\partial H_{\mu_2^k}(x^k)$ readily follows from $\dom\,H_{\mu_2^k}=\dom\,h=\real^n$ following from $(A0)$.
We will show the latter statement.
Let 
$\phi_{k}(x):=
G_{\mu^k_1}(x)-\langle y ^k,x-x^k\rangle.
$
For any $\{z^l\}\subseteq \mathcal{F}$ such that $\|z^l\|\to\infty$, it holds that 
\begin{equation}
\phi_{k}(z^l)=
f_{u^k}(z^l)+H_{\mu^k_2}(z^l)
-\langle y ^k,z^l-x^k\rangle\to \infty\ (l\to\infty),\notag 
\end{equation}
since $H_{\mu^k_2}$ is strongly convex by Assumption\,$(A2)$
and $\{f_{u^k}(z^l)\}$ can be shown to be bounded from below by using Proposition\,\ref{kesukamo} and Lemma\,\ref{lem:1222}.
Hence, $\phi_k$ is coercive and convex in $\mathcal{F}$,
from which we can conclude that
\eqref{subp1} has an optimum.
It remains to show that the optimal value is finite. Notice that Lemma\,\ref{lem:1222} yields $\dom~g=\dom~G_{\mu_1}$ and thus $\dom f=\dom g=\dom~G_{\mu_1}$ holds. Then, from Assumption~$(A3)$, there exists some feasible point $x\in \mathcal{F}$ such that $G_{\mu_1}(x)<\infty$.
Therefore, we have $\min_{x\in\mathcal{F}}G_{\mu_1}(x)-\langle y ^k,x-x^k\rangle<\infty$.
\end{proof}
The next proposition concerns the decrement of the objective value per iteration.
\begin{proposition}\label{prop:1013}
Suppose that Assumption\,$(A1)$ holds.
Then, for any $k\ge 1$, it holds that  
\begin{equation}
f_{u^{k+1}}(x^{k+1})\le f_{u^k}(x^{k})-\frac{\tau}{2}\|{x^{k+1}-x^k}\|^2
+\kappa(\mu_{2}^k-\mu_{2}^{k+1}),\label{eq:1226-1}
\end{equation}
Moreover, the same can be said for the function $\tilde{f}^{\rm cl}$:
\begin{equation}
\tilde{f}^{\rm cl}_{u^{k+1}}(x^{k+1})
\le \tilde{f}^{\rm cl}_{u^k}(x^{k})-\frac{\tau}{2}\|{x^{k+1}-x^k}\|^2+\kappa(\mu_{2}^k-\mu_{2}^{k+1}). \label{eq:1226-2}
\end{equation}
\end{proposition}
\begin{proof}
For simplicity of expression, we consider only the case 
where $\{u^k\}\subseteq \real^2_{++}$. Hence, $\partial H_{\mu_2^k}(x^k)=\nabla H_2(x^k)$ for all $k\ge 0$.
The argument below can be directly extended to the case where $\mu^k_1=0$ or $\mu^k_2=0$ for any $k\ge 0$.

Since $x^{k+1}$ is an optimum of \eqref{bright2} for $k\ge 0$ and $x^k$ is feasible to \eqref{bright2} for $k\ge 1$, 
we have
\begin{equation}
G_{\mu^k_1}(x^{k+1})-\langle y^k, x^{k+1}-x^k\rangle
\le G_{\mu^k_1}(x^{k})-\langle y^k, x^{k}-x^k\rangle= G_{\mu^k_1}(x^k).\label{eq:1222-2259}
\end{equation}
On the other hand, we derive 
$H_{\mu^k_2}(x^{k+1})-H_{\mu^k_2}(x^k)\ge \langle y^k, x^{k+1}-x^k\rangle + \tau\|x^{k+1}-x^k\|^2/2$
from 
$y^k\in\nabla H_{\mu_2^k}(x^{k})$ and 
the strong convexity of $H_{\mu^k_2}$ with a modulus $\tau>0$ (Assumption $(A2)$),
and thus we get 
\begin{equation}
-H_{\mu^k_2}(x^{k+1})+ \langle y^k, x^{k+1}-x^k\rangle
\le -H_{\mu^k_2}(x^k)-\frac{\tau}{2}{\|x^{k+1}-x^k\|^2}.\label{eq:1222-2300}
\end{equation}
Summing up the both sides of \eqref{eq:1222-2259} and \eqref{eq:1222-2300} yields 
$$
G_{\mu_1^k}(x^{k+1})-H_{\mu^k_2}(x^{k+1})
\le G_{\mu_1^{k}}(x^k)-H_{\mu^k_2}(x^k)-\frac{\tau}{2}\|x^{k+1}-x^k\|^2.
$$
Combining this in turn with 
$G_{\mu^{k+1}_1}(x^{k+1})\le G_{\mu^k_1}(x^{k+1})$ 
and $H_{\mu^k_2}(x^{k+1})\le H_{\mu^{k+1}_2}(x^{k+1})+\kappa(\mu_{2}^k-\mu_{2}^{k+1})$ following from $(P2)$ yields \eqref{eq:1226-1}.
By noting that  
$
\tilde{G}^{\rm cl}_{\mu^k_1}(x^{k})=G_{\mu^k_1}(x^k)
$
for any $k\ge 1$ since $x^k$ is feasible to \eqref{prob1},
we further derive from \eqref{eq:1226-1} that
$$
\tilde{G}^{\rm cl}_{\mu^k_1}(x^{k+1})-H_{\mu^k_2}(x^{k+1})\le
\tilde{G}^{\rm cl}_{\mu^k_1}(x^{k})-H_{\mu^k_2}(x^k)-\frac{\tau}{2}\|x^{k+1}-x^k\|^2+\kappa(\mu_{2}^k-\mu_{2}^{k+1}).
$$
Therefore, from $\tilde{f}^{\rm cl}_{u}=\tilde{G}^{\rm cl}_{\mu_1}-H_{\mu_2}$,
we conclude \eqref{eq:1226-2}.
\end{proof}
We are now ready to give the final convergence property of the smoothing SCMIP.
\begin{theorem}\label{th:1013}
Suppose that Assumptions\,$(A1)$--$(A3)$ hold. Then,
\begin{enumerate}
\item[$({\rm i})$] the generated sequences $\{x^k\}$ and $\{y^k\}$ are bounded, and 
\item[$({\rm ii})$]$\lim_{k\to\infty}\|x^{k+1}-x^k\|=0$.
\end{enumerate}
Let arbitrary accumulation points of $\{x^k\}$ and $\{y^k\}$ be $x^{\ast}:=(x^{\ast}_M,x^{\ast}_N)\in \real^M\times \integer^N$ and $y^{\ast}\in \mathbb{R}^n$, respectively. Then, we have
\begin{enumerate}
\item[$({\rm iii})$] $x^{\ast}$ is feasible to \eqref{prob1}, 
\item[$({\rm iv})$] $\{x^k_N\}$ converges to $x^{\ast}_N$ in finitely many iterations, and 
\item[$({\rm v})$] $x^{\ast}$ is a stationary point of DC program\,\eqref{prob3}.
In addition, $y^{\ast}$ is a stationary point of the 
Toland-Singer
dual of \eqref{prob3}.
\end{enumerate}
\end{theorem}
\begin{proof}
\begin{enumerate}
\item[(i)] Let $X_1:=\{(x,u)\in \mathcal{F}\times \real^2 \mid 
f_{u}(x)+\kappa \mu_2\le f_{u^1}(x^1)+\kappa \mu_2^1,\ 0\le u\le u^1\}$ and $X_2:=\{(x,u)\in \mathcal{F}\times \real^2\mid f(x)\le f(x^1)+
\kappa(\mu^1_1+\mu^1_2),\ 0\le u\le u^1\}$.
Notice that $X_2$ is bounded from Assumption~$(A1)$.  
Choose $(x,u)\in X_1$ arbitrarily.
Then, from Lemma\,\ref{lem:1222}, we have 
\begin{align}
f(x)\le f_{u}(x)+\kappa\mu_2\le  f_{u^1}(x^1)+\kappa \mu_2^1\le f(x^1)+\kappa(\mu^1_1+\mu^1_2)\notag 
\end{align}
where
the first inequality follows from Lemma\,\ref{lem:1222}\,\eqref{eq:1222-2}, 
the second one from $(x,u)\in X_1$, and the third by using \eqref{eq:1222-2} again with $x=x^1$ and $u=u^1$.
This together with $0\le u \le u^1$ implies $(x,u)\in X_2$.
Therefore, we have $X_1\subseteq X_2$, and thus $X_1$ is bounded.  
Since $\{(x^k,u^k)\}_{k\ge 1}\subseteq X_1$ from Proposition\,\ref{prop:1013}, we then conclude that $\{x^k\}$ is bounded.
To prove the boundedness of $\{y^k\}$, suppose that $\{y^k\}$ is unbounded for contradiction.
Let $(\bar{x},\bar{y})$ be an arbitrary accumulation point of $\left\{\left(x^k,y^k/\|y^k\|\right)\right\}$.
We may assume that $\left(x^k,y^k/\|y^k\|\right)\to (\bar{x},\bar{y})$ as $k\to\infty$ without loss of generality.
From $y^k\in \partial H_{\mu^k_2}(x^k)$, we have 
\begin{equation}
H_{\mu_2^k}(x^k+\bar{y})-H_{\mu_2^k}(x^k)\ge \langle \bar{y},y^k\rangle.\label{eq:1016-1}
\end{equation}
Then, by Definition\,\ref{def:smoothing}(3) with $\Phi=H$, $\lim_{k\to \infty}H_{\mu_2^k}(x^k+\bar{y})=h(\bar{x}+\bar{y})<\infty$ and $\lim_{k\to \infty}H_{\mu_2^k}(x^k)=h(\bar{x})<\infty$, which together with dividing the both sides of \eqref{eq:1016-1} by $\|y^k\|$ and driving $k$ to $\infty$ entail 
$0\ge \|\bar{y}\|^2$,
which contradicts $\|\bar{y}\|=1$. We thus see that $\{y^k\}$ is bounded. 
\item[(ii)]
By Proposition\,\ref{prop:1013} and Lemma~\ref{lem:1222} with $x=x^{k+1}$ and $u=u^{k+1}$, it holds that 
\begin{equation}
f(x^{k+1})\le
f_{u^{k+1}}(x^{k+1})+\kappa\mu^{k+1}_2\le f_{u^k}(x^{k})+\kappa\mu^k_2 -\frac{\tau}{2}\|x^{k+1}-x^k\|^2<\infty\label{eq:1016-2}
\end{equation}
for any $k\ge 1$, where the last inequality holds since
$
f_{u^k}(x^k)\le f(x^k)+\kappa\mu_1^k
$ follows from Lemma~\ref{lem:1222} and
$
x^{k}\in \dom G_{\mu_1^{k-1}}=\dom g=\dom f\ (k\ge 1)
$ from Proposition~\ref{prop:5-1} under Assumption~$(A3)$.
Then, by noting that $\{f(x^k)\}$ is bounded from below according to Proposition~\ref{kesukamo} and the fact that $x^k\in\mathcal{F}$ for $k\ge 1$,
\eqref{eq:1016-2} implies that $\left\{f_{u^k}(x^{k})+\kappa\mu^k_2\right\}$ is a bounded monotone nonincreasing sequence,
and hence a convergent sequence.
Therefore, $\lim_{k\to\infty}\|x^{k+1}-x^k\|=0$ follows from the second and third inequalities of \eqref{eq:1016-2}.
\item[(iii)] 
By noting $\{x^k\}_{k\ge 1}\subseteq \mathcal{F}=S\cap \left(\real^M\times\integer^N\right)$ and the closedness of $\mathcal{F}$, 
$x^{\ast}\in \mathcal{F}$ is easily derived.
\item[(iv)] 
For contradiction, suppose that 
$\{x^k_N\}$ does not converge in finitely many iterations.
Then,  there exist infinitely many $k$ such that  
$\|x^{k+1}_N-x^k_N\|\ge 1$, since $x^k_N\in \integer^N$ for any $k\ge 1$. 
However, it contradicts the fact of $\lim_{k\to \infty}\|x^k-x^{k+1}\|=0$. 
\item[(v)]
We first prove that $x^{\ast}\in \dom \tilde{g}^{\rm cl}$.
Supposing to the contrary, we have $\tilde{g}^{\rm cl}(x^{\ast})=\infty$.  
Then, from the definition of $\tilde{g}^{\rm cl}$ and feasibility of $x^{\ast}$, 
we obtain $g(x^{\ast})=\tilde{g}^{\rm cl}(x^{\ast})=\infty$, and hence $f(x^{\ast})=\infty$.
On the other hand, 
\eqref{eq:1016-2} and \eqref{eq:1222-2} yield that $f(x^{\ast})=\lim_{k\to \infty}f_{u^k}(x^k)+\kappa\mu_2^{k}<\infty$.
This is a contradiction.

We next show the main claim.
As
$
{\rm ri}\,\dom\tilde{G}_{\mu^k_1}^{\rm cl}\neq \emptyset$ from Lemma\,\ref{lem:1021-2}, we obtain 
\begin{equation}
\partial\left(\tilde{G}^{\rm cl}_{\mu^k_1}(x)-\langle y^k,x-x^k\rangle\right)=
\partial\tilde{G}^{\rm cl}_{\mu^k_1}(x)- y^k.
\notag 
\end{equation}
By using  this equality and 
the fact that $x^{k+1}$ is an optimum of the convex program\,\eqref{bright3}
\begin{align}
&\partial H_{\mu^k_2}(x^k)\ni y^k,\
\partial\tilde{G}_{\mu^k_1}^{\rm cl}(x^{k+1})- y^k\ni 0.\ \label{al:1013}
\end{align} 
Without loss of generality (if necessary by taking a subsequence), 
we can assume that $(x^k,y^k)$ converges to $(x^{\ast},y^{\ast})$ as $k$ tends to $\infty$. 
Since $\lim_{k\to \infty}x^{k+1}=x^{\ast}$ is derived from (ii) and $x^{\ast}\in \dom\tilde{g}^{\rm cl}$ holds from the first-half argument,
Lemma\,\ref{prop:1014} and \eqref{al:1013} yield that 
$\partial\tilde{g}^{\rm cl}(x^{\ast})\ni y^{\ast}$ and $y^{\ast}\in\partial h(x^{\ast})$,
which means $y^{\ast}\in \partial\tilde{g}^{\rm cl}(x^{\ast})\bigcap \partial h(x^{\ast})$.
Therefore, $x^{\ast}\in \partial(\tilde{g}^{\rm cl})^{\ast}(y^{\ast})\bigcap \partial h^{\ast}(y^{\ast})$.
This completes the proof.
\end{enumerate}
\end{proof}
\section{Concluding remarks}\label{sec6}
In this paper, we have considered mixed integer programs having DC objective functions and closed convex constraints.  
For these problems, we have extended the result of Maehara, Marumo, and Murota 
concerning continuous relaxations of discrete DC programs
and obtained a continuous 
DC program whose optimal value is exactly equal to the original one.
We have also proposed two algorithms to solve the obtained relaxed problem. 
The first is a fundamental algorithm based on the DCA, 
which is a well-known  algorithm for continuous DC programs. 
In the second, we incorporate a smoothing method into the first,
so that we can handle nonsmooth functions efficiently.
For both methods, we proved that the generated sequence 
converges to a stationary point under some mild assumptions.

Our contribution can be summarized as follows:
\begin{itemize}
\item 
We have proposed a new framework for solving mixed integer DC programs. These are a wide class of problems containing many mixed integer nonlinear programs
which are notorious as being extremely difficult.  
Although our method involves the 
computationally costly routine of repeatedly solving convex MIPs,  
it still has significant merit, since it provides a practical way of dealing
with these tough problems.
\item 
We have theoretically proved convergence of generated sequences,
thus the solutions provided by our algorithms are stationary points
which have good chances of being the global optimum.
\end{itemize}
We conclude this paper by mentioning that while our method
does not obtain polynomial complexity, there may be some specific problems
for which it is tractable.
For example, Maehara, Marumo and Murota \cite{Maehara2015b} showed that 
their DCA-based algorithm can efficiently solve 
the degree-concentrated spanning tree problem.
The search for such problems is a possible direction for future work.

\end{document}